\documentclass{article}
\usepackage[dvips]{graphicx}
\usepackage{a4wide}%,times}
\usepackage[dvips]{graphics,color}      % usual driver
\usepackage{verbatim}
\usepackage{amsmath}
\usepackage{xspace}
\usepackage{pifont}
\usepackage{graphicx}
\usepackage{amssymb}
\usepackage{epic, eepic}
\usepackage{dsfont}
\usepackage{amssymb}
\usepackage{makeidx}

\usepackage{amsmath,afterpage}
\usepackage{epsf}

\def\0{\mathbf{0}}

\def\eps{\varepsilon}

\def\lam{\lambda}
\def\rr{\rightarrow}
\def\dr{\downarrow}

\def \< {\langle}
\def \> {\rangle}

\def\beqa{\begin{eqnarray}}
\def\eeqa{\end{eqnarray}}
\def\beqas{\begin{eqnarray*}}
\def\eeqas{\end{eqnarray*}}

\newtheorem{theorem}{Theorem}[section]
\newtheorem{lemma}[theorem]{Lemma}

\newtheorem{corollary}[theorem]{Corollary}

\newtheorem{remark}[theorem]{Remark}

\newtheorem{definition}[theorem]{Definition}

\numberwithin{equation}{section}
\newcommand{\old}[1]{{}}

\def\endpf{{\ \hfill\hbox{\vrule width1.0ex height1.0ex}\parfillskip 0pt}}
\newenvironment{proof}{\noindent{\bf Proof:}}{\endpf}

\addtocounter{footnote}{0}

\newcommand{\qed}{\hfill\rule{2mm}{2mm}}

\newcommand{\supp}{\operatorname{supp}}
\newcommand\myeq{\stackrel{\normalfont\mbox{d}}{=}}

\newcommand{\bd}{\begin{displaymath}}
\newcommand{\ed}{\end{displaymath}}
\newcommand{\be}{\begin{equation}}
\newcommand{\ee}{\end{equation}}
\newcommand{\bq}{\begin{eqnarray}}
\newcommand{\eq}{\end{eqnarray}}
\newcommand{\bn}{\begin{eqnarray*}}
\newcommand{\en}{\end{eqnarray*}}

\newcommand{\re}{\mathds{R}}

\title{Pathwise Uniqueness for the Stochastic Heat Equation with H\"{o}lder Continuous Drift and Noise Coefficients}
\author{Leonid Mytnik \qquad \qquad \qquad Eyal Neuman \\ \\ Faculty of Industrial Engineering \\ and Management  \\ Technion - Institute of Technology \\ Haifa 3200 \\ Israel }
\date{}

\begin{document}

\maketitle

\paragraph{Abstract.}
We study the solutions of the stochastic heat equation with multiplicative space-time white noise.
We prove a comparison theorem between the solutions of stochastic heat equations with the same noise coefficient which is H\"{o}lder continuous of index $\gamma>3/4$, and drift coefficients that are Lipschitz continuous.
Later we use the comparison theorem to get sufficient conditions for the pathwise uniqueness for solutions of the stochastic heat equation, when both the white noise and the drift coefficients are H\"{o}lder continuous.

\section{Introduction and main results}
We study the solutions of the stochastic heat equation with space-time white noise. This equation has the form
\be \label{SHE}
\frac{\partial}{\partial t} u(t,x) = \frac{1}{2}\Delta u(t,x) + \sigma(t,x,u(t,x))\dot{W}+b(t,x,u(t,x)) , \ \ t\geq0, \ \ x\in \re.
\ee
Here $\Delta$ denotes the Laplacian and $\sigma(t,x,u),b(t,x,u):\re_{+}\times\re^2\rr \re$ are continuous functions with at most a linear growth in the $u$ variable.
We assume that the noise $\dot{W}$ is a space-time white noise on $\re_{+}\times\re$. Equations like (\ref{SHE}) arise as scaling limits of of critical branching particle systems. For example, in the case where $\sigma(t,x,u)=\sqrt{u}$ and $b= 0$,
such equations describe the evolution in time and space of the density of the classical \emph{super-Brownian motion} (see e.g. Section 3.4 of \cite{Perk2002}). If $b= b(x)\geq 0$ is a continuous deterministic function with compact support and $\sigma(u)=\sqrt{u}$,
then the solution to (\ref{SHE}) arises as scaling limit of critical branching particle systems with immigration and the limit is known as \emph{super-Brownian motion with immigration} $b$. In other words, the density of the super-Brownian motion with immigration $b$ satisfies (\ref{SHE}), $P$-a.s. (see e.g. Section 3.4 of \cite{Perk2002}). \\\\
In this work we consider the pathwise uniqueness for the solution of (\ref{SHE}) where $\sigma$ and $b$ are H\"{o}lder continuous in $u$ and $\dot{W}$ is a space-time Gaussian white noise. More precisely $W $ is a mean zero Gaussian process defined on a filtered probability space $(\Omega,\mathcal{F},\mathcal{F}_t,P)$, where $\mathcal{F}_t$ satisfies the usual hypothesis and we assume that $W$ has the following properties. We denote by
\bd
W_t(\phi)=\int_{0}^{t}\int_{\re}\phi(s,y)W(ds,dy), \ \  t\geq 0,
\ed
the stochastic integral of a function $\phi$ with respect to $W$.
We denote by $ \mathcal{C}_c^\infty(\re_{+}\times\re)$ the space of compactly supported infinitely differentiable functions on $\re_{+}\times\re$.
We assume that $W$ has the following covariance structure
\bd
E(W_t(\phi)W_t(\psi))=\int_{0}^{t}\int_{\re}\phi(s,y)\psi(s,y)dy ds, \ \ t\geq 0,
\ed
for $\phi,\psi\in \mathcal{C}_c^\infty(\re_{+}\times\re)$.
The stochastic heat equation with space-time white noise was studied among many others, by Caba{\~n}a \cite{Cabana70}, Dawson \cite{Dawson72}, \cite{Dawson75}, Krylov and Rozovskii \cite{Krylov77}, \cite{Krilov79}, \cite{Krilov79b}, Funaki \cite{Funaki83}, \cite{Funaki84} and Walsh \cite{Walsh}.
Pathwise uniqueness of the solutions for the stochastic heat equation, when the white noise coefficient $\sigma$ and the drift coefficient are Lipschitz continuous was derived in \cite{Walsh}. In \cite{MP09}, pathwise uniqueness for the solutions of the stochastic heat equation, where the white noise coefficient is H\"{o}lder continuous of index $\gamma >3/4$, and again the drift coefficient $b$ is Lipschitz continuous was established. The $d$-dimensional stochastic heat equation driven by colored Gaussian noise was also extensively studied. Pathwise uniqueness for the solutions of the stochastic heat equation driven by colored Gaussian noise, with H\"{o}lder continuous noise coefficients was studied in \cite{MPS06}. The result in \cite{MPS06} was later improved by Rippl and Sturm in \cite{Rippl-Sturm2013}.
The method of proof in \cite{MP09}, \cite{MPS06} and \cite{Rippl-Sturm2013} is a version of the Yamada-Watanabe argument (see \cite{YW71}) for infinite dimensional stochastic differential equations. \\\\
Pathwise uniqueness of the solutions for the stochastic heat equation (\ref{SHE}) where $b$ is a non-Lipschitz measurable function was studied in \cite{GyongyPardoux1993}, \cite{GyongyPardoux1993b}, \cite{Bally-Gyon-Pard1994}, \cite{Gyon1995}, \cite{Cerrai-DaPrato-Flandoli2013}, \cite{DaPrato-Flandoli-Priola-Rockner2013} among others. In these papers $\sigma$ is a Lipschitz continuous function which satisfies a so-called \emph{non-degeneracy condition}. For example in \cite{Gyon1995},  $\sigma$ satisfies the following non-degeneracy condition
\bq \label{non-deg}
\sigma(t,x,u)\geq \eps >0,  \ \forall (t,x,u)\in \re_{+}\times \re \times \re,
\eq
and $b$ is a $P\times \mathcal{B}(\re)$-measurable random field which also satisfies some integrability conditions.
The idea of proof in \cite{GyongyPardoux1993}, \cite{GyongyPardoux1993b} \cite{Bally-Gyon-Pard1994} and \cite{Gyon1995} is as follows: first proving uniqueness in law by Girsanov's theorem and then using a comparison theorem to get pathwise uniqueness. The proofs also use Malliavin calculus to get estimates on the density of the solutions to (\ref{SHE}) without drift.   \\\\
%Note that non-degeneracy conditions such as (\ref{non-deg}) are also necessary to prove existence and uniqueness of a solution to deterministic second-order ecliptic equations and to Linear evolution equations (see Section 6.2 and 7.1 and 7.2 in \cite{evens}).  \\\\
Before we describe in more detail the known uniqueness results for the stochastic heat equation with space-time white noise, we introduce additional notation and definitions.
\paragraph{Notation.}
For every $E\subset \re$, we denote by $\mathcal{C}(E)$ the space of continuous functions on $E$. In addition, a superscript $k$, (respectively, $\infty$), indicates that functions are in addition $k$ times (respectively, infinitely often), continuously differentiable. A subscript $c$ indicates that they also have compact support. \medskip \\
For $f\in \mathcal{C}(\re)$ set
\be \label{CtemNorm}
\|f\|_\lam=\sup_{x\in \re}|f(x)|e^{-\lam|x|}, \ \lam \in \re, 
\ee
and define $$\mathcal{C}_{tem}:=\{f\in \mathcal{C}(\re), \|f\|_\lam<\infty \ \ \textrm{for every} \ \  \lam>0\}.$$ The topology on this space is induced by the norms $\|\cdot \|_\lam$ for $\lam>0$. \\\\
For $I\subset \re_{+}$ let $\mathcal{C}(I,E)$ be the space of all continuous functions on $I$ taking values in topological space $E$ endowed with the topology of uniform convergence on compact subsets of $I$.
Hence the notation $u \in \mathcal{C}(\re_{+},\mathcal{C}_{tem})$ implies that $u$ is a continuous function on $\re_{+}\times \re$ and
\bq \label{c-tem-sol}
\sup_{t\in[0,T]}\sup_{x\in\re}|u(t,x)|e^{-\lam|x|}< \infty , \ \ \forall  \lam>0, \ T>0.
\eq
In many cases it is possible to show that solutions to (\ref{SHE}) are in $\mathcal{C}(\re_{+},\mathcal{C}_{tem})$.
Let us define a stochastically strong solution to (\ref{SHE}), which is also called a strong solution to (\ref{SHE}).
\begin{definition} (Definition next to Equation (1.5) in \cite{MP09}) \label{Def-Strong-sol}
Let $(\Omega,\mathcal{F},\mathcal{F}_t,P)$ be a probability space and let $W$ be a white noise process defined on $(\Omega,\mathcal{F},\mathcal{F}_t,P)$. Let $\mathcal{F}_t^W\subset \mathcal{F}_t$ be the filtration generated by $W$. A stochastic process $u:\Omega \times \re_{+}\times \re \rr \re $ which is jointly measurable and $\mathcal{F}^W_t$-adapted, is said to be a stochastically strong solution to (\ref{SHE}) with initial condition $u_0$ on $(\Omega,\mathcal{F},\mathcal{F}_t,P)$, if for all $t \geq 0$ and $x\in \re$,
\bq \label{MSHE}
u(t,x)&=&G_{t}u_0(x) + \int_{0}^{t}\int_{\re}G_{t-s}(x-y)\sigma(s,y,u(s,y))W(ds,dy) \nonumber \\
&&+\int_{0}^{t}\int_{\re}G_{t-s}(x-y)b(s,y,u(s,y))dyds, \ \  P-\rm{a.s.}
\eq
\end{definition}
Here
\bn
G_t(x)=\frac{1}{\sqrt{2\pi t}}e^{\frac{-x^2}{2t}}, \ x\in\re , \ t>0,
\en
and $G_{t}f(x)=\int_{\re}G_{t}(x-y)f(y)dy$, for all $f$ such that the integral exists.
\\\\
In this work we study the uniqueness property in the sense of pathwise uniqueness. The definition of pathwise uniqueness is given below.
 \begin{definition} (Definition before Theorem 1.2 in  \cite{MP09})
We say that pathwise uniqueness holds for solutions of (\ref{SHE}) in $\mathcal{C}(\re_{+},\mathcal{C}_{tem})$ if for every deterministic initial condition, $u_0\in \mathcal{C}_{tem}$, any two solutions to (\ref{SHE}) with sample paths a.s. in $\mathcal{C}(\re_{+},\mathcal{C}_{tem})$ are equal with probability $1$.
\end{definition}
%\begin{definition} (Definition in Section 2 of \cite{BassPerkins2011})
%We say that uniqueness in law holds for solutions of (\ref{SHE}) in $\mathcal{C}(\re_{+},\mathcal{C}_{tem})$, if for every deterministic initial condition, $u_0\in \mathcal{C}_{tem}$, any two solutions to (\ref{SHE}), $(u,W)$ and $(\bar{u},\bar{W})$, with sample paths a.s. in $\mathcal{C}(\re_{+},\mathcal{C}_{tem})$, have the same joint law.
%\end{definition}
\paragraph{Convention.} Constants whose values are unimportant and may change from line to line are denoted by $C_i,M_i, \ i=1,2,..$, while constants whose values will be referred to later and appear initially in say, Equation $(i.j)$ are denoted by $C_{(i.j)}$. \\\\
Next we present in more detail some results on pathwise uniqueness for the solutions of (\ref{SHE}) driven by space-time white noise which are relevant to us.
When $\sigma$ and $b$ are Lipschitz continuous, the existence and uniqueness of a strong solution to (\ref{SHE}) in $\mathcal{C}(\re_+,\mathcal{C}_{tem})$ was proved in \cite{Shi94}. The proof uses the standard tools that were developed in \cite{Walsh} for solutions to SPDEs.
In \cite{MP09}, Lipschitz assumptions on $\sigma$ were relaxed and the following conditions were introduced:
for every $T>0$, there exists a constant $C_{(\ref{grow})}(T)>0$ such that for all  $(t,x,u)\in [0,T]\times \re^2$,
\be \label{grow}
|\sigma(t,x,u)|+|b(t,x,u)|\leq C_{(\ref{grow})}(T)(1+|u|).
\ee
Also, for some $\gamma>3/4$ there are $ R_1, R_2>0$ and for all $T>0$ there is an $ R_0(T)$ so that for all $t\in[0,T]$ and all $(x,u,u')\in \re^3$,
\be \label{HolSigmaCon}
|\sigma(t,x,u) - \sigma(t,x,u')| \leq  R_0(T)e^{ R_1|x|}(1+|u|+|u'|)^{ R_2}|u-u'|^\gamma,
\ee
and there is $B>0$ such that for all $(t,x,u,u')\in \re_{+}\times\re^3$,
\be \label{LipsDrift}
|b(t,x,u) - b(t,x,u')| \leq  B|u-u'|.
\ee
%and there is a $B>0$ such that for all $(t,x,u,u')\in \re_{+}\times \re^3$,
%\be \label{HolDriftCon}
%|b(t,x,u) - b(t,x,u')| \leq B|u-u'|.
%\ee
Mytnik and Perkins in \cite{MP09} proved that if $u_0\in \mathcal{C}_{tem}$, and $b,\sigma:\re_{+}\times\re^2\rr\re$ satisfy (\ref{grow}), (\ref{HolSigmaCon}) and (\ref{LipsDrift}), then there exists a unique strong solution of (\ref{SHE}) in $\mathcal{C}(\re_+,\mathcal{C}_{tem})$.
\begin{remark}  \label{MomentRemark}
It was also proved in \cite{MP09} (see Equation (2.25)) that if $b$ and $\sigma$ satisfy (\ref{grow}) then any $\mathcal{C}(\re_+,\mathcal{C}_{tem})$ solution $u$ to (\ref{SHE}) satisfies
\bq \label{Gen-Moment-Cond}
E\big(\sup_{t\in [0,T]}\sup_{x\in\re}|u(t,x)|^pe^{-\lam|x|}\big)< \infty, \ \forall \lam,p >0.
\eq
The bound (\ref{Gen-Moment-Cond}) will be very useful in our proofs later.
\end{remark}
Now we are ready to present our main results. The first result is a comparison theorem for the solutions of (\ref{SHE}) with H\"{o}lder continuous $\sigma$ and a Lipschitz continuous $b$.
The second result of the paper is the pathwise uniqueness for (\ref{SHE}) under assumptions (\ref{grow}) and (\ref{HolSigmaCon}) on $\sigma$ and while relaxing the Lipschitz assumption on $b$.
\medskip \\
In what follows we assume that the drift coefficient $b(t,x,u,\omega):\re_{+}\times\re^2\times \Omega\rr \re$ in (\ref{SHE}) is an $\{\mathcal{F}^{W}_{t}\}$~-predictable function. We further assume that the constants $C(T)$ and $B$ in (\ref{grow}), and (\ref{LipsDrift}) do not dependent on $\omega$. The dependence of $b$ in $\omega$ is often suppressed in our notation for the sake of readability.
\begin{theorem}\label{Theorem-Holder-Comp}
Assume that $b_i:\re_{+}\times\re^2 \times \Omega \rr \re, \ (i=1,2)$ and $\sigma:\re_{+}\times\re^2\rr\re,$ satisfy (\ref{grow}), (\ref{HolSigmaCon}), (\ref{LipsDrift}), $P$-a.s.
Let $u^i(t,\cdot)$ be a $\mathcal{C}_{tem}$-valued solution of (\ref{MSHE}) associated with the coefficients $\sigma$ and $b_i$, having initial condition $u^i(0,\cdot)=u^{i}_0(\cdot)\in \mathcal{C}_{tem}, \ i=1,2$. Suppose further that
\bq
\sigma(t,x,u) \textrm{ and } b_i(t,x,u,\omega) \ (i=1,2) \  \textrm{ are continuous in } (x,u), \ P-\rm{a.s.} ,
\eq
\bq \label{cond-geq}
b_1(t,x,u,\omega) \leq b_2(t,x,u,\omega), \ \forall t\geq 0, \ x\in \re, \ u\in \re, \ P-\rm{a.s.},
\eq
and
\bq \label{cond-init}
u^1_0(x)\leq u^2_0(x), \ \forall x\in \re.
\eq
Then,
\bq \label{comp-prop}
P(u^2(t,\cdot)\geq u^1(t,\cdot), \textrm { for every } t\geq 0)=1.
\eq
\end{theorem}
\begin{remark}
Comparison theorems for the stochastic heat equations when the white noise and drift coefficients $\sigma$ and $b$
 are Lipschitz continuous were proved in \cite{Mueller91} and \cite{Shi94} among others. A weaker version of a comparison theorem was proved in Proposition 3.1 in \cite{MMQ2010} for the case of a non-Lipschitz $\sigma$. It was proved in \cite{MMQ2010} that there exists a probability space on which there is a white noise $\dot W$ such that (\ref{comp-prop}) holds. Note that in Theorem \ref{Theorem-Holder-Comp} the probability space and the white noise are specified in advance.
\end{remark}
Before we state our main results we introduce some notation and recall Girsanov's theorem for the white noise process.
\paragraph{Notation.}
Let $W$ be a space-time white noise. Denote by $\{\mathcal{F}^W_t\}_{t\geq 0}$ the complete, right continuous filtration of $W$.
Fix $T>0$ and let $\{Z(s,x):  (s,x)\in [0,T] \times \re \}$ be an $\re$-valued random field that is adapted to $\mathcal{F}_t^W$ and satisfies
\bq \label{Cond-girs1}
\int_{0}^T\int_{\re}Z(s,x,\omega)^2 dx ds < \infty, \ P-\rm{a.s.}
\eq
For $t\in[0,T]$, define
\bq \label{L-t}
L_t=\exp{\bigg(\int_{0}^t\int_{\re}Z(s,x)W(ds , dx)-\frac{1}{2}\int_{0}^t\int_{\re}Z(s,x)^2 dxds \bigg)}.
\eq
The following version of Girsanov's theorem for the white noise process is Theorem 10.14 in \cite{Daprato92}.
Other versions can be found in Proposition 1.6 \cite{NualartPardoux1994} or Theorem 3.4 in \cite{Mueller-Dalang2013}.
\begin{theorem}[Theorem 10.2.1 in \cite{Daprato92}]\label{girsanov}
If $\{Z(s,x):  (s,x)\in [0,T] \times \re \}$ is such that $E(L_T)=1$ then
\bn
\widetilde W(dt , dx) = Z(t,x)dt dx +W(dt , dx), \ t\in [0,T], \ x\in \re,
\en
is a space-time white noise under the probability measure $Q$, where $Q$ is defined by
\bq \label{dQ-dP}
\frac{dQ}{dP}\bigg|_{\mathcal F^{W}_{T}}=L_T. 
\eq
Here $L_{T}$ is the Radon-Nikodym derivative of $Q$ with respect to $P$ restricted to $\mathcal F^{W}_{T}$.
\end{theorem}
\begin{remark} \label{rem-nov} 
The assumption that $E(L_T)=1$ is often replaced by the assumption that $\{L_t\}_{t\in[0,T]}$ is a martingale with respect to the filtration $\mathcal{F}_t^W$.
Such assumption is satisfied if, for example Novikov's condition holds:
\bq \label{nov-cond}
E\bigg(\exp{\bigg(\frac{1}{2}\int_{0}^t\int_{\re}Z(s,x)^2 dx ds\bigg)}\bigg) <\infty,  \ \textrm{for every} \  t\in [0,T],
\eq
%It is well known that if (\ref{nov-cond}) holds, then $\{L_t\}_{t\in[0,T]}$ is a martingale with respect to the filtration $\mathcal{F}_t^W$
see for example Proposition 10.17 in \cite{Daprato92}.
%(See Theorem 3.5.3 in \cite{Daprato92}).
\end{remark}
Before we state our main results we will need some additional definitions.  \medskip \\
Let $(\Omega,\mathcal{F},\mathcal{F}_t,P)$ be a probability space and let $W$ be a white noise process defined on $(\Omega,\mathcal{F},\mathcal{F}_t,P)$. Let $u_0\in \mathcal{C}_{tem}$. Denote by $\mathcal{S}_{u_0}^W$ the class of strong $\mathcal{C}(\re_{+},\mathcal{C}_{tem})$-solutions to (\ref{SHE}). \medskip \\
Now we will give our basic assumptions on the drift coefficient $b$ in (\ref{SHE}).
\paragraph{Assumption A.}
We assume that there exists an $\{\mathcal{F}^{W}_{t}\}$~-predictable function $Z(t,x,u,\omega):\re_{+}\times \re^2 \times \Omega \rr \re$ such that
\bq \label{zsx}
b(t,x,u,\omega)=Z(t,x,u,\omega){\sigma(t,x,u)}, \ \forall t\in [0,T], \ x\in \re, \ u\in \re, \ P-\rm{a.s.}
\eq
The following theorem gives a sufficient condition for the pathwise uniqueness of a solution to (\ref{SHE}) when both the white noise and drift coefficients are non-Lipschitz. This is the main result of this paper. The existence of a weak solution to this equation under less restrictive assumptions on $\sigma$ and $b$ was proved in \cite{MP09}.
\begin{theorem}\label{Thm-UNIQ}
Let $\dot{W}$ be a space-time white noise. Let $u(0,\cdot)\in \mathcal{C}_{tem}(\re)$. Let $b: \re_{+}\times \re^2\times \Omega \rr \re$ and $\sigma: \re_{+}\times \re^2\rr \re$ be continous in $(x,u)$ and satisfy (\ref{grow}) and Assumption A , $P$-a.s. Let $\sigma$ satisfy (\ref{HolSigmaCon}) for some $\gamma>3/4$. Assume that for every $u\in \mathcal{S}_{u_0}^W$, $Z(t,x,u(t,x))$ from (\ref{zsx}) satisfies (\ref{Cond-girs1}).  Then pathwise uniqueness holds for the solutions of (\ref{SHE}) with sample paths a.s. in $\mathcal{C}(\re_{+},\mathcal{C}_{tem})$.
\end{theorem}
One of the applications of Theorem \ref{Thm-UNIQ} is the pathwise uniqueness of solutions to (\ref{SHE}) in the special case where:
\bq \label{power-uniq}
\sigma(u) =|u|^{p} \textrm{ and } b(u)=-|u|^{q}, \textrm{ for some }\ 3/4< p<q \leq 1.
\eq
\paragraph{Notation.}
For $f\in \mathcal{C}(\re)$ set
$$\mathcal{C}_{rap}:=\{f\in \mathcal{C}(\re), \|f\|_\lam<\infty \ \ \textrm{for every} \ \  \lam<0\}.$$ The topology on this space is induced by the norms $\|\cdot \|_\lam$ for $ \lam<0$. 
Denote by $\mathcal{C}_{rap}^{+}$ (respectively $\mathcal{C}_{tem}^{+}$)  the set of nonnegative functions in $\mathcal{C}_{rap}$ (respectively $\mathcal{C}_{tem}$).  \medskip \\ 
The existence of a stochastically weak $\mathcal{C}(\re_{+},\mathcal{C}^{+}_{tem})$ solution to (\ref{SHE}) was proved in Theorem 1.1 of \cite{Shi94} for a larger class of $b$ and $\sigma$ which also includes our example in (\ref{power-uniq}). 
One can easily show that if $b$ and $\sigma$ satisfy (\ref{power-uniq}) and 
$u(0,\cdot)\in \mathcal{C}^{+}_{rap}$, then any $\mathcal{C}(\re_{+},\mathcal{C}^{+}_{tem})$ solution to (\ref{SHE}) is also in $\mathcal{C}(\re_{+},\mathcal{C}^{+}_{rap})$ (the proof follows the same lines as the proof of Theorem 2.5 in \cite{Shi94}). Under the assumption in (\ref{power-uniq}) we have
\bq \label{zsx1}
Z(s,x,u(s,x)):=|u(s,x)|^{q-p}, \ x\in \re, \ s \geq 0, \ \textrm{for every } u\in \mathcal{S}_{u_0}^W,
\eq
and we get that (\ref{Cond-girs1}) is satisfied for every $u\in \mathcal{S}_{u_0}^W$. From the discussion above and Theorem \ref{Thm-UNIQ} we get the following corollary.
\begin{corollary} \label{corr-l2}
Let $\dot{W}$ be a space-time white noise. Let $u(0,\cdot)$ be in $\mathcal{C}^{+}_{rap}$. Assume that $\sigma$ and $b$ are as in (\ref{power-uniq}). Then pathwise uniqueness holds for the solutions to (\ref{SHE}) with sample paths a.s. in $\mathcal{C}(\re_{+},\mathcal{C}^{+}_{tem})$.
\end{corollary}
\begin{remark}
 Recall that one of the necessary conditions for the pathwise uniqueness theorems in \cite{GyongyPardoux1993}, \cite{GyongyPardoux1993b}, \cite{Bally-Gyon-Pard1994} and \cite{Gyon1995}, is the non-degeneracy condition (\ref{non-deg}). One of the by products of Corollary \ref{corr-l2} is that (\ref{non-deg}) is not a necessary condition for pathwise uniqueness. 
 \end{remark}
The rest of this paper is devoted to the proofs of Theorems \ref{Theorem-Holder-Comp} and \ref{Thm-UNIQ}. In Section \ref{Sec-copm} we prove Theorem \ref{Theorem-Holder-Comp}. In Section \ref{Sec-uniq} we prove Theorem \ref{Thm-UNIQ}.
\section{Proof of Theorem \ref{Theorem-Holder-Comp}} \label{Sec-copm}
This section is devoted to the proof of Theorem \ref{Theorem-Holder-Comp}. The proof of a comparison theorem for SDEs with non-Lipschitz noise and drift coefficients was carried out by Nakao in \cite{Nakao1983}. The proof of Theorem \ref{Theorem-Holder-Comp} uses ideas from Nakao's proof.
%Before we prove Theorem \ref{Theorem-Holder-Comp} we recall the following proposition from Section 2 of \cite{MP09}.
%\begin{proposition} \label{parameter}
%Let $u_0\in C_{tem}$. Let $\sigma$ and $b$ be continuous functions satisfying (\ref{grow}). Then any solution $u \in C(\re_{+},C_{tem})$ to (\ref{SHE}) satisfies for any $T,\lam>0$ and $p\in (0,\infty)$,
%\bd
%E\bigg(\sup_{0\leq t \leq T}\sup_{x\in \re}|u(t,x)|^pe^{-\lam|x|}\bigg) < \infty.
%\ed
%\end{proposition}
First, let us introduce the following notation.
\paragraph{Notation.} We denote by $(i,j)( b,\sigma)$, equation $(i,j)$ with drift function $ b$ and white noise coefficient $ \sigma $. \medskip  \\
We will also use a more general notion of strong solution that was introduced in \cite{Kurtz2007}. Let $S_1$ and $S_2$ be Polish spaces and let $\Gamma:S_1\times S_2\rr \re$ be a Borel measurable function. Let $Y$ be an $S_2$-valued random variable with distribution $\nu$. We are interested in the solution $(X,Y)$  to the equation
\bq \label{gen-eqn}
\Gamma(X,Y) =0.
\eq
\begin{definition}[Definition 2.1 in \cite{Kurtz2007}]\label{def-sol-Kurtz}
A solution $(X,Y)$ to (\ref{gen-eqn}) is called a strong solution if there exists a Borel measurable function $F:S_2\rr S_1$ such that $X=F(Y)$, $P$-a.s.
\end{definition}
In our case $Y$ is the white noise process $W$ and we only consider strong solutions that are adapted to $\{\mathcal{F}_t\}_{t\geq 0}$, the natural filtration of the white noise. 
%The following remark relates to minor extensions of Theorems 1.1 and 1.2 in \cite{MP09}.
\begin{remark}  \label{extension-myt}
The existence and uniqueness results of Theorems 1.1 and 1.2 in \cite{MP09}, still hold if we assume that the drift coefficient $b(t,x,u,\omega):\re_{+}\times\re^2\times \Omega\rr \re$ in (\ref{SHE}) is an $\mathcal{F}_{t}$-predictable function which satisfy (\ref{grow}) and (\ref{LipsDrift}), $P$-a.s., where constants $C(T),B,$ in (\ref{grow}) and (\ref{LipsDrift}) are independent of $\omega$.
The proof of this generalisation follows directly from the proof in Section 8 of \cite{MP09} hence it is omitted. We can also get (\ref {Gen-Moment-Cond}) for the solutions of (\ref{SHE}) when $b$ and $\sigma$ satisfy the assumptions in Remark \ref{MomentRemark} and where $b$ is an $\mathcal{F}_{t}$-predictable function as above. 
\end{remark}
In order to prove Theorem \ref{Theorem-Holder-Comp}, we need the following additional notation. 
\paragraph{Notation.} 
With a slight abuse of notation set
\bq \label{brownian-sheet} 
W(t,x)=
\left\{
\begin{array}{ll}
\int_{0}^{t}\int_{y=0}^{x} W(dy,ds),  \  x \geq 0, \\\\
-\int_{0}^{t}\int_{y=x}^{0} W(dy,ds), \  x < 0.
\end{array}
\right.
\eq 
Note that $t \mapsto	 W(t,\cdot) \in  \mathcal{C}_{tem}$, $P$-a.s. This can be easily verified by checking the conditions of Lemma 6.3(i) in \cite{Shi94}. 
%Indeed, the continuity of $\{W(t,x)\}_{t\geq0,x\in \re}$ can be easily verified by Kolmogorov's continuity theorem (see e.g. Theorem 4.3 in Chapter 1 of \cite{SPDEbook09}). Moreover, for every fixed $t>0$, $\{W(t,\cdot)\}_{x\geq 0}$ is a driftless Brownian motion. From the law of iterated logarithms it easily follows that  $\sup_{x>0}e^{-\lam x}|W(t,x)|<\infty$, $P$-a.s., for every $\lam>0$. 
%By applying the same argument to $x<0$ we get that $t \mapsto	 W(t,\cdot) \in  \mathcal{C}_{tem}$, $P$-a.s.

\paragraph{Proof of Theorem \ref{Theorem-Holder-Comp}:}
Let $u^i$ be two solutions of (\ref{SHE})$(\sigma,b_i), \ i=1,2$, with the same white noise on $(\Omega,\mathcal{F},\mathcal{F}_t,P)$, with sample paths in $\mathcal{C}(\re_{+},\mathcal{C}_{tem})$ a.s. Here $\{\mathcal{F}_t\}_{t\geq 0}$ is the natural filtration of the white noise process $W$. Assume that $u^i, \ i=1,2,$ have the deterministic initial conditions
\bq \label{init-cond-eq}
u^i(0,\cdot)=u^{i}_0(\cdot)\in \mathcal{C}_{tem}, \ i=1,2, 
\eq
which satisfy (\ref{cond-init}).
Note that by Theorem 1.3 in \cite{MP09}, there exists a unique strong solution to (\ref{SHE})$(\sigma,b_i)$, for $i=1,2$.   
By Definition \ref{def-sol-Kurtz}, this means that for each $i=1,2$, there exists a unique measurable function $F_i:\mathcal{C}(\re_{+},\mathcal{C}_{tem})\rr \mathcal{C}(\re_{+},\mathcal{C}_{tem})$,
%$F_i:S_2\rr S_1$, where $S_2$, the space of the white noise $W$ is  $\mathcal{C}(\re_{+},\mathcal{C}_{tem})$, and $S_1$, the space of the solution $u_i$, is also $\mathcal{C}(\re_{+},\mathcal{C}_{tem})$.
such that $u_i=F_i(W), \ P$-a.s.$, \ i=1,2$ are $\mathcal{F}_t$-adapted (if $\tilde u_i$ is any other strong solution to (\ref{SHE})$(\sigma,b_i)$, then $\tilde u_i=F_i(W), \ P$-a.s. as well). \medskip \\
Note that in order to prove Theorem  \ref{Theorem-Holder-Comp} it is sufficient to show
\bd
F_1(W)(t,x)\leq F_2(W)(t,x), \ \forall t\geq 0, \ x\in\re,  \ P-\rm{a.s.}
\ed
Let $\Psi_n\in \mathcal{C}_c^{\infty}$ be a symmetric function so that $0\leq \Psi_n\leq 1$, $\|\Psi'_n\|_{\infty}\leq 1$, $\Psi_n(x)=1$ if $|x|\leq n$ and $\Psi_n(x)=0$ if $|x|\geq n+2$.
Let $$\sigma^n(t,x,u)=\int_{\re}\sigma(t,x,u')G_{2^{-n}}(u-u')\Psi_n(u')du'.$$ By the proof of Theorem 1.1 in \cite{MP09}, $\sigma^n(t,x,u)$ is Lipschitz continuous in $u$ and satisfies (\ref{grow}) for every $n\in \mathds{N}$. \\
Let
\be \label{MSHEn}
\left\{
\begin{array}{ll}
u_i^n(t,x)=G_{t}u_0(x) + \int_{0}^{t}\int_{\re}G_{t-s}(x-y)\sigma^n(s,y,u^{n}(s,y))W(ds,dy)   \\   \\
\quad \quad  \quad  \   \ + \int_{0}^{t}\int_{\re}G_{t-s}(x-y)b_i(s,y,u^{n}(s,y))dyds, \ x\in \re, \ t\geq 0, \   P-\rm{a.s.}, \ i=1,2,  \\   \\
u_i^n(0,x)= u_0(x), \ \ x\in \re, \ i=1,2. \medskip \\ 
\end{array}
\right.
\ee
From Theorem 2.2 in \cite{Shi94} we get that for every $i$ and $n$ there exists unique solution $u_i^n$ to (\ref{MSHEn}).
Let $Z_n=(u_1^n,u_2^n,W)$. Now argue as in the proof of Theorem 1.1 in \cite{MP09} that the family of laws $P^{Z_n}$  is tight in $\mathcal{C}(\re_{+},\mathcal{C}_{tem})^{3}$. Let $\{n_k\}_{k\geq 0}$ be a subsequence such that 
$\{Z_{n_k}\}_{k\geq 0}$ converges weakly to $Z=( u_1, u_2, W)$. 
By Skorohod's theorem there exists a probability space $(\widetilde\Omega,\widetilde {\mathcal{F}},\widetilde{\mathcal{F}}_t,\widetilde P)$ on which the sequence of processes  
$\{\widetilde Z_{n_k}\}_{k \geq 0}=\{(\tilde u_1^{n_k}, \tilde u_2^{n_k},\widetilde {W}^{n_{k}})\}_{k\geq 0} \myeq \{Z_{n_k}\}_{k \geq 0}$ is defined and converges $\widetilde {P}$-a.s. to $\widetilde Z=(\tilde u_1,\tilde u_2,\widetilde W)  \myeq ( u_1, u_2,W) $, hence $\widetilde Z$ is also in $\mathcal{C}(\re_{+},\mathcal{C}_{tem})^{3}$. Here $ \dot { \widetilde{W}}$ is a space-time white noise on $(\widetilde\Omega,\widetilde {\mathcal{F}},\widetilde{\mathcal{F}}_t,\widetilde P)$ and the Brownian sheet $\widetilde{W}$ is defined analogously to $W$ in (\ref{brownian-sheet}). Note that just as in the proof of Theorem 1.1 in \cite{MP09}, $\tilde u_i,\ i=1,2$ are weak solutions to  (\ref{SHE})$(\sigma,b_i)$ with $\widetilde W$.  \medskip \\
Since $b_1(t,x,u)\leq b_2(t,x,u)$, $P$-a.s. and $u^1_0(x)\leq u^2_0(x)$, we have from Corollary 2.4 in \cite{Shi94} that $u_1^{n_k}(t,x)\leq u_2^{n_k}(t,x)$ for all $t>0,\ x\in\re$, $P$-a.s., for every $k=1,2,\dots$
We get that
\bq \label{rr}
\tilde u_1(t,x)\leq \tilde u_2(t,x), \  \forall t>0,\ x\in\re,  \ \widetilde P-\rm{a.s.}
\eq
Now recall that by Theorem 1.3 in \cite{MP09},  $u_{i}$ are unique strong solutions to (\ref{SHE})$(\sigma,b_i)$ with $\widetilde W$. Hence, from (\ref{rr}) we get
\bd
F_1(\widetilde W)(t,x)\leq F_2(\widetilde W)(t,x), \ \forall t\geq 0, \ x\in\re,  \ \widetilde P-\rm{a.s.}
\ed
Therefore we conclude that
\bd
F_1( W)(t,x)\leq F_2( W)(t,x),  \ \forall t\geq 0, \ x\in\re,  \ P-\rm{a.s.}
\ed
and we are done. \qed

\section{Proof of Theorems \ref{Thm-UNIQ}} \label{Sec-uniq}
\subsection{Auxiliary Lemmas} \label{Auxiliary-Lemmas}
In this section we prove a few auxiliary lemmas that will be used in the proof of Theorem \ref{Thm-UNIQ}. Before we start with the proofs, we recall the distributional form of (\ref{SHE}). \medskip \\
Let $u$ be a solution of (\ref{SHE}) on $(\Omega,\mathcal{F},\mathcal{F}_t,P)$ with sample paths in $\mathcal{C}(\re_{+},\mathcal{C}_{tem})$, $P$-a.s. and the initial condition $u(0)=u_0\in C_{tem}$.
By Theorem 2.1 in \cite{Shi94}, (\ref{MSHE}) is equivalent to the distributional form of (\ref{SHE}). That is, 
\bq \label{mildForm}
\langle u(t),\phi \rangle &=& \langle u_0,\phi \rangle + \frac{1}{2}\int_{0}^{t}\big(\langle u(s),\frac{1}{2}\Delta \phi \rangle +\langle b(s,\cdot,u(s,\cdot), \phi \rangle \big) ds \nonumber \\
&&+ \int_{0}^{t}\int_{\re} \sigma(s,x,u(s,x))\phi(x)W(ds , dx), \ \ \forall t \geq 0, \ P-\rm{a.s.}, \ \forall \phi\in C_c^{\infty}(\re).
\eq
\paragraph{Convention:} Let $f:\re_{+}\times \re^{2}  \times \Omega \rr \re$ be an 
$\{\mathcal F_{t}\}_{t\geq 0}$-predictable function. We say that $f$ is continuous in the $x$ variable if 
\bn
x  \mapsto f(t,x,u) \ \textrm{ is continuous} \ \  \forall(t,u)\in \re_{+}\times \re , \ P-\rm{a.s.},
\en
and in the $u$ variable if 
\bq
u  \mapsto f(t,x,u) \ \textrm{ is continuous} \ \  \forall(t,x)\in \re_{+}\times \re, \ P-\rm{a.s.} 
\eq
Let $f_{n}:\re_{+}\times \re^{2}  \times \Omega \rr \re, \ n\geq 1,$ be a sequence of $\{\mathcal F_{t}\}_{t\geq 0}$-predictable functions. We say that $\{f_{n}\}_{n\geq 1}$ is monotone increasing sequence of functions if 
\bq \label{inc}
f_{n}(t,x,u) \leq f_{n+1}(t,x,u), \  \forall(t,x,u)\in \re_{+}\times \re^{2}, \ n\geq 1, \ P-\rm{a.s.} 
\eq
The definition of a decreasing sequence of functions is completely analogous to (\ref{inc}). We say that $\{f_{n}\}_{n\geq 1}$ is a monotone sequence of functions if it is either increasing or decreasing sequence.  \medskip \\
Before we start with the proof of Theorem \ref{Thm-UNIQ}, we will also need the following crucial lemma.
\begin{lemma}\label{lemma-aux}
Let $b_n:\re_{+}\times \re^{2}  \times \Omega \rr \re$, $n\geq 1$, be a sequence of monotone $\{\mathcal F_{t}\}_{t\geq 0}$-predictable functions such that for every $n\geq 1$, $b_{n}$ is continuous in $u$ the variable, and satisfy for some $T>0$,
\bq \label{unif-b-n}
\sup_{n\in \mathds{N}}|b_n(t,x,u,\omega)|\leq C(T)(1+|u|), \ \forall (t,x,u) \in [0,T]\times \re^2, \ P-\rm{a.s.}
\eq
Assume that
\bq \label{b-point} 
\lim _{n\rr \infty } b_n(t,x,u) = b(t,x,u),   \  \forall (t,x,u) \in [0,T] \times \re^{2} , \ P-\rm{a.s.},  
\eq
where $b:\re_{+}\times \re^{2} \times \Omega \rr \re$ is an $\{\mathcal F_{t}\}_{t\geq 0}$-predictable function which is continuous in the $u$ variable. Let $\sigma:\re_{+}\times \re \times \re  \rr \re$ be a continuous function in $u$ variable which satisfies $(\ref{grow})$. Assume that (\ref{mildForm})$(b_n,\sigma)$ admits jointly measurable $\mathcal{F}_t$-adapted solution $u_n$ such that
\bq \label{Moment-Cond}
E\big(\sup_{n}\sup_{t\in [0,T]}\sup_{x\in\re}|u_n(t,x)|^2e^{-\lam|x|}\big)< \infty,
\eq
and
\bq \label{bl1}
u_n(t,x)\rr u(t,x), \  \forall (t,x)\in [0,T]\times \re, \  P-\rm{a.s.},
\eq
as $n\rr\infty$, where $u:\re_{+}\times\re \rr \re$ is also a jointly measurable and $\mathcal{F}_t$-adapted process. Then $u(t,\cdot)$ satisfies (\ref{mildForm})$(b,\sigma)$ for any $t\in[0,T]$. 
\end{lemma}
\begin{proof}
Note that from (\ref{Moment-Cond}) and (\ref{bl1}) it immediately follows that
\bq \label{Moment-Cond2}
E\big(\sup_{t\in [0,T]}\sup_{x\in\re}|u(t,x)|^2e^{-\lam|x|}\big)< \infty.
\eq
Let us show that $u$ solves (\ref{mildForm})$(b,\sigma)$. Let $\phi\in \mathcal{C}^{\infty}_c$, and fix $K>0$ such that 
\be \label{supp}
\supp(\phi)\subset [-K,K].
\ee
From (\ref{Moment-Cond})--(\ref{Moment-Cond2}) and dominated convergence we get
\bn
\lim_{n\rr \infty}E\big(|\langle u_n(t),\phi \rangle  - \langle u(t),\phi \rangle|\big)=0, \ \forall t\in[0,T], 
\en
and
\bn
\lim_{n\rr \infty}E \bigg( \bigg| \frac{1}{2}\int_{0}^{t}\langle u_n(s)-u(s),\frac{1}{2}\Delta \phi \rangle  ds \bigg| \bigg)=0, \ \forall t\in[0,T].
\en
Let us show the convergence of the drift term. Note that for every $t\in[0,T]$, 
\bq \label{fr0}
E\bigg(\bigg|\int_{0}^{t}\langle  b_n(s,\cdot,u_n(s,\cdot))-b(s,\cdot,u(s,\cdot)),\phi \rangle ds \bigg| \bigg)&\leq&
E\bigg(\bigg| \int_{0}^{t}\langle  b_n(s,\cdot,u_{n}(s,\cdot))-b(s,\cdot,u_{n}(s,\cdot)),\phi \rangle ds \bigg|\bigg)\nonumber \\
 &&+E\bigg(\bigg|\int_{0}^{t}\langle  b(s,\cdot,u_n(s,\cdot))-b(s,\cdot,u(s,\cdot)),\phi \rangle ds \bigg| \bigg)  \nonumber \\
 &&=:I_1+I_2.
\eq
From (\ref{unif-b-n}) and since $\{b_n\}_{n\geq 1}$ converges pointwise to $b$, $P$-a.s., we have
\bq \label{unif-b}
|b(t,x,u,\omega)|\leq C(T)(1+|u|), \ \forall (t,x,u) \in [0,T]\times \re^2, \ P-\rm{a.s.}
\eq
From (\ref{unif-b-n}) and (\ref{unif-b})  we have
\bq \label{rff1}
&& \sup_{s\in [0,T]}\sup_{x\in \re}\big(|b_{n}(s,x,u_{n}(s,x))-b(s,x,u_{n}(s,x))||\phi(x)|\big)  \nonumber \\
&&\leq  C(T)\sup_{s\in [0,T]}\sup_{x\in \re}\big((1+|u_{n}(s,x)|)|\phi(x)| \big)\nonumber \\
&&\leq  C(T)\bigg(1+\sup_{s\in [0,T]}\sup_{x\in \re}(|u_{n}(s,x)|e^{-\lam|x|})\bigg)
\bigg(\sup_{x\in \re}e^{\lam|x|}|\phi(x)| \bigg),  \ P-\rm{a.s.}
\eq
Since $\{b_{n}\}_{n\geq1}$ is a monotone sequence of continuous functions in the $u$ variable and $b$ is continuous in the $u$ variable, we get from (\ref{b-point}) that  for every $0<M<\infty$ we have 
\bq \label{b-unif-con} 
\lim_{n\rr\infty }\sup_{u\in [-M,M]}\big|b_{n}(s,x,u) -b(s,x,u)\big| =0, \  \forall (s,x)\in [0,T]\times \re,  \  P-\rm{a.s.}.
\eq
Use (\ref{Moment-Cond2}), (\ref{supp}), (\ref{rff1}), (\ref{b-unif-con}) and the dominated convergence theorem again to get,
\bq \label{fr1}
\lim_{n\rr\infty} I_1=0.
\eq
For $I_2$ we can use again (\ref{unif-b-n}) and (\ref{unif-b}) to get
\bq \label{rff2}
&& \sup_{s\in [0,T]}\sup_{x\in \re}\big(|b(s,x,u_n(s,x))-b(s,x,u(s,x))||\phi(x)| \big) \\
&&\leq  C(T)\sup_{s\in [0,T]}\sup_{x\in \re}\big((1+|u_n(s,x)|+|u(s,x)|)|\phi(x)|\big)  \nonumber \\
&&\leq C(T)\bigg(1+\sup_{s\in [0,T]}\sup_{x\in \re}|(u_n(s,x)|+|u(s,x)|)|e^{-\lam|x|}\bigg)\bigg(\sup_{x\in \re}e^{\lam|x|}|\phi(x)| \bigg), \ P-\rm{a.s.} \nonumber
\eq
Use (\ref{Moment-Cond}), (\ref{Moment-Cond2}), (\ref{supp}), (\ref{rff2}) and the dominated convergence theorem again to get,
\bq \label{rf2}
&&\lim_{n\rr\infty}E\bigg(\int_{0}^{t}\int_{\re}|b(s,x,u_n(s,x))-b(s,x,u(s,x))||\phi(x)| dx ds \bigg) \nonumber \\
&&=E\bigg(\int_{0}^{t}\int_{\re}\lim_{n\rr\infty}|b(s,x,u_n(s,x))-b(s,x,u(s,x))||\phi(x)| dx ds \bigg) \nonumber \\
&&=0,
\eq
where the last equality follows from (\ref{bl1}) and the continuity of $b$ in the $u$ variable. \medskip \\
From (\ref{rff2}) and (\ref{rf2}) we get
\bq \label{fr3}
\lim_{n\rr\infty} I_2=0.
\eq
From (\ref{fr0}), (\ref{fr1}) and (\ref{fr3}) we get
\bn
\lim_{n\rr\infty} E\bigg(\bigg| \int_{0}^{t}\langle  b_n(s,\cdot,u_n(s,\cdot))-b(s,\cdot,u(s,\cdot)),\phi \rangle ds \bigg|\bigg)=0,\ \forall t\in[0,T].
\en
Now let us handle the stochastic integral term. Denote by
\bq
M_t^n:=\int_{0}^{t}\int_{\re} \sigma(s,x,u_n(s,x))\phi(x)W(ds , dx), \ t>0,  \ n\geq 0,
\eq
and
\bq
M_t:=\int_{0}^{t}\int_{\re} \sigma(s,x,u(s,x))\phi(x)W(ds , dx), \ t>0, \ n\geq 0.
\eq
From (\ref{grow}) we have,
\bq \label{ineq61}
[(\sigma(s,x,u_n(s,x))-\sigma(s,x,u(s,x)))\phi(x)]^2 \leq C(T)(1+|u_n(t,x)|^2+|u(t,x)|^2)\phi(x)^2
\eq
From (\ref{supp}) we immediately get
\bq \label{ineq62}
&&(1+|u_n(t,x)|^2+|u(t,x)|^2)\phi(x)^2  \\
&&\leq  C(T)\bigg(1+\sup_{s\in [0,T]}\sup_{y\in \re}\big(|u_n(s,y)|^2+|u(s,y)|^2)e^{-|y|}
\big)\bigg)\bigg(\sup_{y\in \re} \phi(y)^2e^{|y|}\bigg) \mathds{1}_{\{x\in[-K,K]\}}, \ \forall t\in[0,T], \ x\in\re. \nonumber
\eq 
From (\ref{Moment-Cond}), (\ref{Moment-Cond2}), (\ref{ineq61}) and (\ref{ineq62}), dominated convergence and the continuity of $\sigma$ in the $u$ variable we have
\bq \label{ineq51}
\lim_{n\rr \infty} <M^n_\cdot-M_\cdot>_t &=&\lim_{n\rr \infty} E\bigg(\int_{0}^{T}\int_{\re} [\sigma(s,x,u_n(s,x))-\sigma(s,x,u(s,x))]^2\phi(x)^2dxds\bigg)    \\
&=&E\bigg(\int_{0}^{T}\int_{\re}\lim_{n\rr \infty}[\sigma(s,x,u_n(s,x))-\sigma(s,x,u(s,x))]^2\phi(x)^2dxds \bigg)\nonumber \\
&= & 0,\  \forall t\in[0,T].\nonumber
\eq 
From Theorem 3.1(1) in \cite{Kunita90} we get that the sequence of square integrable martingales $\{M^n\}_{n\geq 1}$ converges to $M$ in $L^2$, where $M$ is also a square integrable martingale, that is
\bn \label{varineq}
\lim_{n\rr\infty} E\bigg[\bigg(\int_{0}^{t}\int_{\re} \sigma(s,x,u_n(s,x))\phi(x)W(ds , dx)-\int_{0}^{t}\int_{\re} \sigma(s,x,u(s,x))\phi(x)W(ds , dx)\bigg)^2\bigg]=0,\ \forall t\in[0,T].
\en
and we are done.
\end{proof} \\ \\
The following lemma is a special case of equation (2.4e) in \cite{rosen}. 
\begin{lemma} \label{rosen}
There exist constants $C_{\ref{rosen}},C'_{\ref{rosen}}>0$ such that,
\bn
\big|G_{t}(x)-G_{t}(y) \big| &\leq& C_{\ref{rosen}}|x-y|t^{-1}\Big(e^{-C'_{\ref{rosen}}x^{2}/t}+e^{-C'_{\ref{rosen}}y^{2}/t}\Big), \  \forall t>0, \ x,y\in \re. 
 \en
 \end{lemma}  
 We will also need the following lemma.
\begin{lemma} \label{Lemma-Shiga}
Let $T>0$ and $\lam\geq 0$. There exists a constant $C_{\ref{Lemma-Shiga}}(\lam,T)>0$ such that,
\bq \label{int-g}
 \int_{0}^{t\vee t'}\int_{\re} e^{\lam|y|}(G_{t'-s}(x'-y)-G_{t-s}(x-y))^2 dy ds &\leq& C_{\ref{Lemma-Shiga}}(T,\lam) e^{\lam|x|}e^{\lam|x-x'|}(|t'-t|^{1/2}+|x'-x|),  \nonumber \\
 && \forall 0 \leq t,t' \leq T, \ x,x'\in \re, 
\eq
where $G_{t}(x-y)=0$ for $t\leq 0$.
\end{lemma}
\begin{proof}
Lemma \ref{Lemma-Shiga} appears in \cite{Shi94} (see Lemma 6.2(i)) for the case where $\lam=0$. More details on proof of (\ref{int-g}) when $\lam=0$, are given in the proof of Theorem 6.7 in Chapter 1.6 of \cite{SPDEbook09}. The proof for the case where $\lam>0$ follows the same lines and hence it is omitted.
\end{proof}
\subsection{Proof of Theorem \ref{Thm-UNIQ}}
\paragraph{Step 1. Construction of a Solution}
In this step we construct a solution to (\ref{mildForm}). Later we will show that this solution is the unique solution of (\ref{SHE}).
\medskip \\
We assume that $\sigma$ satisfies (\ref{HolSigmaCon}) and both $\sigma$ and $b$ are continous in $(x,u)$ and satisfy (\ref{grow}).
Let $\Psi_n$ be as in the proof of Theorem \ref{Theorem-Holder-Comp}. For any $m\in\mathds{N}$, define
\bq \label{b-m}
b_m(t,x,u,\omega)=\int_{\re}b(t,x,u',\omega)G_{2^{-m}}(u-u')\Psi_m(u')du', \ P-\rm{a.s.}
\eq
Let
\bq \label{b-n-k}
\tilde b_{n,k}&:=&\wedge_{m=n}^{k}b_m, \ n\leq k, \nonumber \\
\tilde b_n&:=&\wedge_{m=n}^{\infty}b_m.
\eq
Fix an arbitrary $T>0$. As a direct consequence of (\ref{grow}) we have,
\bq \label{b-m-grow}
|b_m(t,x,u)|\leq C_{\ref{grow}}(T)(1+|u|),  \ \forall (t,x,u)\in [0,T]\times\re\times \re, \ m\in \mathds{N}, \ P-\rm{a.s.}
\eq
Again from (\ref{grow}) and Lemma \ref{rosen} we have 
\bq \label{b-m-lip}
|b_m(t,x,u)-b_m(t,x,u')|\leq C(m)|u-u'|, \ \forall (t,x,u)\in [0,T]\times\re\times \re , \ m\in \mathds{N}, \ P-\rm{a.s.} 
\eq
From (\ref{b-m}) follows that 
\bq \label{pointwise}
b_{m}(t,x,u)\rr b(t,x,u),
\eq
pointwise for any $(t,x,u)\in [0,T] \times \re\times \re$, $P$-a.s. 
From (\ref{b-n-k}), (\ref{b-m-lip}) and (\ref{pointwise}) we can easily get $\tilde b_{n,k}$ is Lipschitz in $u$ uniformly with respect to $(t,x)\in [0,T]\times \re $ and
\bq \label{bn-cnov}
\tilde b_{n,k} &\dr&  \tilde b_n, \ \textrm{as } k\rr \infty, \nonumber \\
\tilde b_n &\uparrow & b, \ \textrm{as } n\rr \infty,
\eq
for any $(t,x,u)\in [0,T]\times \re^{2}, \ P$-a.s.  \medskip \\
By Theorems 1.2 and 1.3 in \cite{MP09}, there exists a unique strong $\mathcal{C}(\re_{+},\mathcal{C}_{tem})$-valued solution to (\ref{SHE})$(\tilde b_{n,k},\sigma)$.
Denote by $\tilde u_{n,k}$ the solution of (\ref{SHE})$(\tilde b_{n,k},\sigma)$. From Theorem \ref{Theorem-Holder-Comp} we get that the sequence $\{\tilde u_{n,k}\}$ decreases with $k$, hence it has a $P$-a.s. pointwise limit
\bq \label{un-const}
u_n:=\lim_{k\rr \infty}\tilde u_{n,k}.
\eq
Note that $u_n$ is also jointly measurable, $\mathcal{F}_t$-adapted process since it is an infimum of the jointly measurable, $\mathcal{F}_t$-adapted processes $\{\tilde u_{n,k}\}_{k\geq 1}$.
Denote by $\bar b (u):= C_{\ref{grow}}(T)(1+|u|)$ and note that $\bar b$ is Lipschitz uniformly in $u$ and satisfies trivially (\ref{grow}).
From Theorems 1.2 and 1.3 in \cite{MP09} we get that there exists a unique strong $\mathcal C(\re_{+},\mathcal{C}_{tem})$ solution to (\ref{SHE})$(\bar b,\sigma)$ and to (\ref{SHE})$(-\bar b ,\sigma)$. Denote by $\bar u$ ($\underline{u}$, respectively) the solution to (\ref{SHE})$(\bar b,\sigma)$ (the solution to (\ref{SHE})$(-\bar b ,\sigma)$, respectively).  \medskip \\
From Theorem \ref{Theorem-Holder-Comp} and (\ref{b-m-grow}) we have
\bq \label{ineq11}
\underline{u}(t,x) \leq \tilde u_{n,k}(t,x)\leq \bar{u}(t,x), \ \forall x\in\re, \ t\in[0,T], \ k, n\in \mathds{N}, \ k \geq n, \ P-\rm{a.s.}
\eq
From the fact that $\underline{u}(t,x),  \bar{u}(t,x)\in \mathcal C(\re_{+},\mathcal{C}_{tem})$ and (\ref{Gen-Moment-Cond}) we get
\bq \label{ineq111}
E\big(\sup_{t\in [0,T]}\sup_{x\in\re}(|\underline u(t,x)|^p+|\bar u(t,x)|^p)e^{-\lam|x|}\big)< \infty, \ \forall \lam, \ p>0.
\eq
Furthermore, from (\ref{ineq11}) and (\ref{ineq111}) we get
\bq \label{fact}
E\big(\sup_{n}\sup_{k\geq n}\sup_{t\in [0,T]}\sup_{x\in\re}(|\tilde u_{n,k}(t,x)|^p+|u_n(t,x)|^p)e^{-\lam|x|}\big)< \infty, \ \forall \lam, \ p>0.
\eq
Note that for every $n\geq 1,$ the sequence $\{\tilde b_{n,k}\}_{k\geq n}$ is uniformly bounded and equicontinuous in the $u$ variable, therefore from (\ref{bn-cnov}) we get that $\tilde b_{n}$ is continuous in the $u$ variable. From (\ref{b-m-grow}), (\ref{bn-cnov}), (\ref{un-const}), (\ref{fact}) and Lemma \ref{lemma-aux}, we get that $u_n$ solves (\ref{mildForm})$(\tilde b_n,\sigma)$.  \medskip \\
We would like to construct our solution to (\ref{mildForm})$(\sigma,b)$ as the limit of $u_n$, as $n\rr \infty$. From Theorem \ref{Theorem-Holder-Comp} we get
\bq \label{ineq1}
\tilde u_{n,k}\geq \tilde u_{m,k}, \ \forall m\leq n\leq k,
\eq
and hence $u_n$ increases as $n$ increases. Since $\{u_n\}_{n\geq 1}$ is an increasing sequence of jointly measurable, $\mathcal{F}_t$-adapted processes, we get that $u_n$ converges pointwise:
\bq \label{u-lim}
u(t,x):=\lim_{n\rr \infty} u_n(t,x),\  \forall (t,x)\in [0,T]\times \re, \ P-\rm{a.s.}
\eq
and that $u$ is also jointly measurable, $\mathcal{F}_t$-adapted process.
Therefore by (\ref{b-m-grow}), (\ref{bn-cnov}), (\ref{fact}), (\ref{u-lim}) and Lemma \ref{lemma-aux} we get that $u$ solves (\ref{mildForm})$(b,\sigma)$. From (\ref{fact}) and (\ref{u-lim}) we get that
\bq \label{fact22}
E\big(\sup_{t\in [0,T]}\sup_{x\in\re}|u(t,x)|^pe^{-\lam|x|}\big)< \infty, \ \forall \lam, \ p>0.
\eq
\paragraph{Step 2. Continuity of the Constructed Strong Solution}
In this step we prove that the strong solution constructed in Step 1 has a modification which is jointly continuous. Note that the continuity of $u$ together with (\ref{fact22}) implies that $u\in \mathcal{C}(\re_{+},\mathcal{C}_{tem})$. The proof uses ideas from the proof of Theorem 2.2 in \cite{Shi94}. \medskip \\
Let
\bq  \label{Xnk}
X_{n,k}(t,x)=\int_{0}^{t}\int_{\re}G_{t-s}(x-y)\sigma(s,y,\tilde u_{n,k}(s,y))W(ds,dy), \ \forall x\in\re, \  t\geq 0, \ n,k\in \mathds{N}, \ n \leq k.
\eq
Let $\lam>0$ be an arbitrary constant. Apply Burkholder's inequality and (\ref{grow}) to (\ref{Xnk}) to get for every $p>1$,
\bq \label{rr1}
&& E((X_{n,k}(t',x')-X_{n,k}(t,x))^{2p}) \nonumber \\
&&\leq  C(p,T)E\bigg[\bigg(\int_{0}^{t\vee t'}\int_{\re}e^{2\lam|y|}(G_{t'-s}(x'-y)-G_{t-s}(x-y))^2(1+e^{-\lam|y|}[\tilde u_{n,k}(s,y)]^{2})dy ds\bigg)^{p}\bigg],  \nonumber \\
&&\quad \forall 0 \leq t,t' \leq T, \ x,x'\in \re, \  n,k\in \mathds{N}, \ n \leq k.
\eq
Use (\ref{fact}) and Lemma \ref{Lemma-Shiga} on (\ref{rr1}) to get that there exists a constant $C_{(\ref{rr2})}(\lam,p,T)>0$ independent of $n,k$ such that
\bq \label{rr2}
E((X_{n,k}(t',x')-X_{n,k}(t,x))^{2p}) &\leq&  C_{(\ref{rr2})}(\lam,p,T)e^{2\lam p|x|}e^{2\lam p|x-x'|}(|t'-t|^{p/2}+|x'-x|^p), \nonumber \\
 && \forall 0 \leq t,t' \leq T, \ x,x'\in \re, \  n,k\in \mathds{N}, \ n \leq k.
\eq
Let
\bq  \label{Ynk}
Y_{n,k}(t,x)=\int_{0}^{t}\int_{\re}G_{t-s}(x-y)b_{n,k}(s,y,\tilde u_{n,k}(s,y))dy ds, \ \forall x\in\re, \  t\geq 0, \ n,k\in \mathds{N}, \ n \leq k.
\eq
From (\ref{b-n-k}) and (\ref{b-m-grow}) we have
\bq \label{grow-bnk}
|b_{n,k}(t,x,u)|\leq C_{(\ref{grow})}(T)(1+|u|),  \ \forall (x,u)\in \re\times \re, \ t\in[0,T] , \ n,k\in \mathds{N}, \ n \leq k, \ P-\rm{a.s.}
\eq
Use (\ref{grow-bnk}), (\ref{fact}), Jensen's inequality and Lemma \ref{Lemma-Shiga} to get that there exists a constant $C_{(\ref{rr4})}(\lam,p,T)>0$ independent of $n,k$ such that
\bq \label{rr4}
&& E((Y_{n,k}(t',x')-Y_{n,k}(t,x))^{2p})   \\
&&\leq   C(p,T)\big(1+E\big(\sup_{k\geq n}\sup_{s\in [0,T]}\sup_{y\in\re}e^{-2p\lam |y|}|\tilde u_{n,k}(s,y)|^{2p}\big)\big) \nonumber \\
&&\quad \times \bigg(\int_{0}^{t\vee t'}\int_{\re}e^{2\lam|y|}|G_{t'-s}(x'-y)-G_{t-s}(x-y)|e^{-\lam  |y|}dy ds\bigg)^{2p} \nonumber \\
&&\leq   C(\lam,p,T)\bigg(\int_{0}^{t\vee t'}\int_{\re}e^{4\lam|y|}(G_{t'-s}(x'-y)-G_{t-s}(x-y))^2 e^{-\lam|y|}dy ds\bigg)^{p} \nonumber \\
&&\leq C_{(\ref{rr4})}(\lam,p,T)e^{3\lam p |x|}e^{3\lam p |x-x'|}(|t'-t|^{p/2}+|x'-x|^p),\
 \forall 0 \leq t,t' \leq T, \ x,x'\in \re, \  n,k\in \mathds{N}, \ n \leq k. \nonumber
\eq
Recall that $u_{0}\in \mathcal{C}_{tem}$. Then from Jensen's inequality and Lemma \ref{Lemma-Shiga} we get that there exists $C_{(\ref{rr5})}(\lam,p,T)>0$ such that,
\bq \label{rr5}
|G_{t}u_0(x)-G_{t'}u_0(x')|^{2p}&=&\Bigg|\int_{\re}\big(G_{t}(x-y)-G_{t'}(x'-y)\big)e^{2\lam|y|}u_{0}(y)e^{-2\lam|y|}dy\Bigg|^{2p} \\
&\leq&\Big(\sup_{y\in \re}|u_{0}(y)|e^{-\lam|y|}\Big)^{2p}\bigg|\int_{\re}\big |G_{t}(x-y)-G_{t'}(x'-y)\big|e^{2\lam|y|}e^{-\lam|y|}dy\bigg|^{2p}\nonumber \\
&\leq&C(\lam,p)\bigg|\int_{\re}\big(G_{t}(x-y)-G_{t'}(x'-y)\big)^{2}e^{4\lam|y|}e^{-\lam|y|}dy\bigg|^{p}\nonumber \\
&\leq&C_{(\ref{rr5})}(\lam,p,T)e^{3\lam p |x|}e^{3\lam p |x-x'|}(|t'-t|^{p/2}+|x'-x|^p),\
 \forall 0 \leq t,t' \leq T, \ x,x'\in \re. \nonumber
\eq
Recall that $\tilde u_{n,k}$ for any $n\leq k$ is a solution to (\ref{MSHE})$(\tilde b_{n,k},\sigma)$, therefore,
\bq \label{rr6}
\tilde u_{n,k}(t,x)=G_{t}u_0(x)+Y_{n,k}(t,x)+X_{n,k}(t,x), \  \forall x\in \re, \ t\geq 0, \  n,k\in \mathds{N}, \ n \leq k.
\eq
From (\ref{fact}), (\ref{rr2}), (\ref{rr4}), (\ref{rr5}) and dominated convergence (applied twice) we have
\bq \label{rr7}
 &&E((u(t',x')-u(t,x))^{2p}) \nonumber \\
&&=\lim_{n\rr\infty} \lim_{k\rr\infty} E((\tilde u_{n,k}(t',x')- \tilde u_{n,k}(t,x))^{2p}) \nonumber \\
&&\leq  (C_{(\ref{rr2})}(\lam, p,T)+C_{(\ref{rr4})}(\lam,p,T)+C_{(\ref{rr5})}(\lam,p,T))e^{3p\lam|x|}e^{3p\lam|x-x'|}(|t'-t|^{p/2}+|x'-x|^p),\nonumber \\
&& \quad \ \forall 0 \leq t,t' \leq T, \ x,x'\in \re.
\eq
From (\ref{rr7}) and Kolmogorov's continuity theorem (see e.g. Theorem 4.3 in Chapter 1 of \cite{SPDEbook09}) we get that there exists a continuous modification of $u$, and together with (\ref{fact22}) we get that $ u\in \mathcal{C}(\re_{+},\mathcal{C}_{tem})$, \ $P$-a.s.
\paragraph{Step 3. Pathwise Uniqueness of the Solution}
In this step we prove the pathwise uniqueness for $\mathcal{C}(\re_{+},\mathcal{C}_{tem})$-solutions of (\ref{SHE})$(b,\sigma)$, starting from the same initial condition.   \medskip \\
Recall that $Z$ was defined in (\ref{zsx}). We assume throughout this step that for every strong $\mathcal{C}(\re_{+},\mathcal{C}_{tem})$ solution $v$ of (\ref{SHE})$(b,\sigma)$  there exists $K>0$ such that
\bq \label{z-asump}
\int_{0}^{T}\int_{\re}Z^2(t,x,v(t,x))dxds \leq K, \ P-\rm{a.s.}
\eq
This assumption is relaxed in Step 4. Note however that that by Remark \ref{rem-nov}, assumption (\ref{z-asump}) essures that $L=\{L_{t}\}_{t \in [0,T]}$ of (\ref{L-t}) is a martingale so that in particular $E\big(L_{t}\big)=1$ for every $t\in [0,T]$.
\medskip   \\
First we show that uniqueness in law for $\mathcal{C}(\re_{+},\mathcal{C}_{tem})$ solutions of (\ref{SHE})$(b,\sigma)$ holds. Let $v$ be a strong $\mathcal{C}(\re_{+},\mathcal{C}_{tem})$ solution to (\ref{SHE})$(b,\sigma)$ on $(\Omega,\mathcal{F},\mathcal{F}_t,P)$.
%\be \label{v-sol}
%\frac{\partial}{\partial t} v(t,x) = \frac{1}{2}\Delta v(t,x) + \sigma(t,x,v(t,x))\dot{W} + b(t,x,v(t,x)), \ \ t\geq0, \ \ x\in \re,
%\ee
%and $v(0,\cdot)=u(0,\cdot)$.
Note that by Theorem \ref{girsanov} we have for every $\phi\in \mathcal{C}_c^\infty$ 
\bq  \label{g-trns}
&&\int_{0}^{t}\int_{\re} b(s,x,v(s,x)) \phi(x)dx ds+ \int_{0}^{t}\int_{\re} \sigma(s,x,v(s,x))\phi(x)W(ds , dx) \nonumber  \\
&&=\int_{0}^{t}\int_{\re}Z(s,x)\sigma(s,x,v(s,x))\phi(x)dx ds+ \int_{0}^{t}\int_{\re} \sigma(s,x,v(s,x))\phi(x)W(ds , dx) \nonumber  \\
&&=\int_{0}^{t}\int_{\re}\sigma(s,x,v(s,x))\phi(x)\widetilde W(ds,dx) ,  \ \ \forall t \geq 0, \ P-\rm{a.s.},
\eq
where $\widetilde W$ is a white noise process under the probability measure $Q$, defined by (\ref{dQ-dP}).  
It follows from (\ref{g-trns}) that $v$ solves (\ref{mildForm})$(0,\sigma)$ under the measure $Q$. 
\paragraph{Notation.} We denote by $E^{ P}$ the expectation with respect to the probability measure $ P$. \medskip \\  
Note that from (\ref{L-t}) and (\ref{z-asump}) we have 
\bq \label{Lt-sec-moment}
E^{P}(L_t^{2})&=&E^{P}\bigg[\exp{\bigg(2\int_{0}^t\int_{\re}Z(s,x,v(s,x))W(ds , dx)-2\int_{0}^t\int_{\re}Z(s,x,v(s,x))^2 dxds \bigg)} \nonumber \\
\quad && \times \exp{\bigg(\int_{0}^t\int_{\re}Z(s,x,v(s,x))^2 dxds \bigg)}\bigg] \nonumber \\
&\leq& C(K)E^{P}\bigg[\exp{\bigg(2\int_{0}^t\int_{\re}Z(s,x,v(s,x))W(ds , dx)-\frac{1}{2}\int_{0}^t\int_{\re}\big[2Z(s,x,v(s,x))\big]^2 dxds \bigg)}\bigg] \nonumber \\
&\leq& C(K),
\eq
where in the last two inequalities we have used (\ref{z-asump}), as well as Remark \ref{rem-nov} in the last inequality. 
By Theorem \ref{girsanov} we can use (\ref{dQ-dP}) together with the Cauchy-Schwarz inequality and (\ref{Lt-sec-moment}) to get for every $p>0$, 
%\bn \label{mom-L-bound} 
%\sup_{t\in [0,T]}\sup_{x\in\re}E^{Q}\big(|u(t,x)|^pe^{-\lam|x|}\big) &=&\sup_{t\in [0,T]}\sup_{x\in\re}E^{P}\big(L_{T}|u(t,x)|^pe^{-\lam|x|}\big)  \nonumber \\
%&\leq& \sup_{t\in [0,T]}\sup_{x\in\re}\bigg(E^{P}\big(|u(t,x)|^{2p}e^{-2\lam|x|}\big)\bigg)^{{1/2}}  \bigg(E^{P}\big(L^{2}_{T}\big)\bigg)^{{1/2}} \nonumber \\
%&<& \infty, \ \forall \lam,p >0.
%\en
\bq \label{Cont1} 
E^{Q}\big(|v(t,x)-v(t',x')|^{2p}\big) &=&E^{P}\big(L_{T}|v(t,x)-v(t',x')|^{2p}\big)  \\
&\leq&\bigg(E^{P}\big(|v(t,x)-v(t',x')|^{4p}\big)\bigg)^{{1/2}}  \bigg(E^{P}\big(L^{2}_{T}\big)\bigg)^{{1/2}} \nonumber \\
&\leq &C(K)\bigg(E^{P}\big(|v(t,x)-v(t',x')|^{4p}\big)\bigg)^{{1/2}}, \ \forall 0\leq t,t' \leq T, \ x,x'\in\re, \ |x-x'|\leq 1 . \nonumber
\eq
Now follow the same lines as in Step 2 with $v$ instead of $\tilde u_{n,k}$ and by using (\ref{grow}) and (\ref{Gen-Moment-Cond}) instead of (\ref{grow-bnk}) and (\ref{fact}).  \medskip \\
From the preceding paragraph we get that for every $\lam>0$ and $p>2$ there exists $C(T,\lam,p)>0$ such that 
\bq \label{det} 
E^{P}\big(|v(t,x)-v(t',x')|^{4p}\big) &\leq& C(T,\lam,p)(|x-x'|^{2p}+|t-t'|^{p})e^{6\lam p |x|},  \nonumber \\
&&\ \forall 0\leq t,t' \leq T, \ x,x'\in\re, \ |x-x'|\leq 1.
\eq
From (\ref{det}) we get for every $\lam>0$ and $p>2$,
\bq \label{Cont2} 
\bigg(E^{P}\big(|v(t,x)-v(t',x')|^{4p}\big)\bigg)^{{1/2}} &\leq& C(T,\lam,p)(|x-x'|^{p}+|t-t'|^{p/2})e^{3\lam p |x|},  \nonumber \\
&&\ \forall 0\leq t,t' \leq T, \ x,x'\in\re, \ |x-x'|\leq 1.
\eq
From (\ref{Cont1}), (\ref{Cont2}) and Lemma 6.3(i) in \cite{Shi94} we get that under the measure $Q$, $v$ has a $\mathcal C(\re_{+},\mathcal C_{tem})$ version which satisfies (\ref{SHE})$(0,\sigma)$. 
Note that from Theorem \ref{girsanov} we also get that $Q$ is absolutely continuous 
with respect to $P$ restricted to $\mathcal{F}^{W}_{T}$, and so together with (\ref{z-asump}) we have
\bq \label{z-asump-Q}
\int_{0}^{T}\int_{\re}Z^2(t,x,v(t,x))dxds \leq K, \ Q-\rm{a.s.}
\eq
We can use (\ref{z-asump-Q}), Remark  \ref{rem-nov} and repeat the same lines as in (\ref{g-trns}) to get that 
\bq \label{dP-dQ}
\frac{dP}{dQ}\bigg|_{\mathcal F^{\widetilde W }_{T}}=\widetilde L_T,
\eq
where 
\bn \label{tilde-L-t}
\widetilde L_t=\exp{\bigg(-\int_{0}^t\int_{\re}Z(s,x,v(s,x))W(ds , dx)+\frac{1}{2}\int_{0}^t\int_{\re}Z(s,x,v(s,x))^2 dxds \bigg)}.
\en
From Theorem 1.3 in \cite{MP09} we get the uniqueness in law for $\mathcal C(\re_{+},\mathcal C_{tem})$ solutions of (\ref{SHE})$(0,\sigma)$. Use this uniqueness in law together with (\ref{dP-dQ}) to get that the law of any strong $\mathcal C(\re_{+},\mathcal C_{tem})$ solution of (\ref{SHE})$(b,\sigma)$ is uniquely determined by the law of $v$ under the measure $Q$ and by $\widetilde L_{T}$,  therefore the uniqueness in law for (\ref{SHE})$(b,\sigma)$ follows. From Theorem 3.14 in \cite{Kurtz2007} we get that the pathwise uniqueness for the solutions of (\ref{SHE})$(b,\sigma)$ follows from the existence of a strong solution $u$ together with the uniqueness in law for (\ref{SHE})$(b,\sigma)$.
\paragraph{Step 4. Pathwise Uniqueness in the General Case}
In this step we assume that $b$ and $\sigma$ satisfy the assumptions of Theorem \ref{Thm-UNIQ}.
Let $v^i, \ i=1,2$ be two strong $\mathcal{C}(\re_{+},\mathcal{C}_{tem})$ solutions of (\ref{SHE})$(b,\sigma)$. Let $K>0$ and define
\bn
T_K&=&\inf\bigg\{t\geq 0:\bigg( \int_{0}^{t}\int_{\re}Z^2(s,x,v^1(s,x)) dxds \bigg)\vee \bigg(\int_{0}^{t}\int_{\re}Z^2(s,x,v^2(s,x)) dxds \bigg)>K\bigg\}.
\en
Denote by
\bn
b_K(t,x,u,\omega)=b(t,x,u,\omega)\mathds{1}(t\leq T_K),
\en
and
\bn
Z_K(s,x,u,\omega)=Z(s,x,u,\omega)\mathds{1}(t \leq T_K), \ i=1,2.
\en
Here $b_K$ and $Z_K$ satisfy (\ref{zsx}) and (\ref{z-asump}). Note that the restrictions on $v^i, \ i=1,2,$ to $[0,T_K]\times \re$ are the restrictions of the unique solution to (\ref{SHE})$(b_K,\sigma)$.
By the hypothesis of Theorem \ref{Thm-UNIQ} we have $T_K\rr \infty$, $P$-a.s. when $K\rr \infty$, and the pathwise uniqueness follows.   \qed

\section{Acknowledgments}
We would like thank an anonymous referee for the careful reading of the manuscript,
and for a number of useful comments and suggestions that improved the exposition.
%\begin{thebibliography}{10}

\def\cprime{$'$} \def\cprime{$'$}

%\bibliographystyle{plain}
%\printindex
%\bibliography{comp}

%\bibliography{Multifractals}

\begin{thebibliography}{10}

\bibitem{Bally-Gyon-Pard1994}
V.~Bally, I.~Gyongy, and E.~Pardoux.
\newblock White noise driven parabolic {SPDE}s with measurable drift.
\newblock {\em Journal of Functional Analysis}, 120(2):484 -- 510, 1994.

\bibitem{Cabana70}
E.~M. Caba{\~n}a.
\newblock The vibrating string forced by white noise.
\newblock {\em Z. Wahrscheinlichkeitstheorie und Verw. Gebiete}, 15:111--130,
  1970.

\bibitem{Cerrai-DaPrato-Flandoli2013}
S.~Cerrai, G.~Da~Prato, and F.~Flandoli.
\newblock Pathwise uniqueness for stochastic reaction-diffusion equations in
  banach spaces with an {H}\"{o}lder drift component.
\newblock {\em Stochastic Partial Differential Equations: Analysis and
  Computations}, 1(3):507--551, 2013.

\bibitem{DaPrato-Flandoli-Priola-Rockner2013}
G.~Da~Prato, F.~Flandoli, E.~Priola, and M.~R{\"o}ckner.
\newblock Strong uniqueness for stochastic evolution equations in {H}ilbert
  spaces perturbed by a bounded measurable drift.
\newblock {\em Ann. Probab.}, 41(5):3306--3344, 2013.

\bibitem{Daprato92}
G.~Da~Prato and J.~Zabczyk.
\newblock {\em Stochastic equations in infinite dimensions}, volume~44 of {\em
  Encyclopedia of Mathematics and its Applications}.
\newblock Cambridge University Press, Cambridge, 1992.

\bibitem{SPDEbook09}
R.C. Dalang, D.~Khoshnevisan, C.~Mueller, D.~Nualart, and Y.~Xiao.
\newblock {\em A minicourse on stochastic partial differential equations},
  volume 1962 of {\em Lecture Notes in Mathematics}.
\newblock Springer-Verlag, Berlin, 2009.
\newblock Held at the University of Utah, Salt Lake City, UT, May 8--19, 2006,
  Edited by Khoshnevisan and Firas Rassoul-Agha.
  
\bibitem{Mueller-Dalang2013}
R.~C. Dalang and C.~Mueller.
\newblock Multiple points of the brownian sheet in critical dimensions.
\newblock {\em (preprint) arXiv:1303.0403 [math.PR]}, 2013.

\bibitem{Dawson72}
D.~A. Dawson.
\newblock Stochastic evolution equations.
\newblock {\em Math. Biosci.}, 15:287--316, 1972.

\bibitem{Dawson75}
D.~A. Dawson.
\newblock Stochastic evolution equations and related measure processes.
\newblock {\em J. Multivariate Anal.}, 5:1--52, 1975.

\bibitem{Funaki84}
T.~Funaki.
\newblock Random motion of strings and stochastic differential equations on the
  space {$C([0,1],{\bf R}^d)$}.
\newblock In {\em Stochastic analysis ({K}atata/{K}yoto, 1982)}, volume~32 of
  {\em North-Holland Math. Library}, pages 121--133. North-Holland, Amsterdam,
  1984.

\bibitem{Gyon1995}
I.~Gy\"{o}ngy.
\newblock On non-degenerate quasi-linear stochastic partial differential
  equations.
\newblock {\em Potential Analysis}, 4(2):157--171, 1995.

\bibitem{GyongyPardoux1993}
I.~Gy\"{o}ngy and E.~Pardoux.
\newblock On quasi-linear stochastic partial differential equations.
\newblock {\em Probability Theory and Related Fields}, 94(4):413--425, 1993.

\bibitem{GyongyPardoux1993b}
I.~Gy\"{o}ngy and E.~Pardoux.
\newblock On the regularization effect of space-time white noise on
  quasi-linear parabolic partial differential equations.
\newblock {\em Probability Theory and Related Fields}, 97(1-2):211--229, 1993.

\bibitem{Funaki83}
A.~Inoue and T.~Funaki.
\newblock On a new derivation of the {N}avier-{S}tokes equation.
\newblock {\em Comm. Math. Phys.}, 65(1):83--90, 1979.

\bibitem{Krylov77}
N.~V. Krylov and B.~L. Rozovskii.
\newblock The {C}auchy problem for linear stochastic partial differential
  equations.
\newblock {\em Izv. Akad. Nauk SSSR Ser. Mat.}, 41(6):1329--1347, 1448, 1977.

\bibitem{Krilov79b}
N.~V. Krylov and B.~L. Rozovskii.
\newblock It\^{o} equations in {B}anach spaces and strongly parabolic
  stochastic partial differential equations.
\newblock {\em Dokl. Akad. Nauk SSSR}, 249(2):285--289, 1979.

\bibitem{Krilov79}
N.~V. Krylov and B.~L. Rozovskii.
\newblock Stochastic evolution equations.
\newblock In {\em Current problems in mathematics, {V}ol. 14 ({R}ussian)},
  pages 71--147, 256. Akad. Nauk SSSR, Vsesoyuz. Inst. Nauchn. i Tekhn.
  Informatsii, Moscow, 1979.

\bibitem{Kunita90}
H.~Kunita.
\newblock {\em Stochastic flows and stochastic differential equations},
  volume~24 of {\em Cambridge Studies in Advanced Mathematics}.
\newblock Cambridge University Press, Cambridge, 1990.

\bibitem{Kurtz2007}
T.~G. Kurtz.
\newblock The {Y}amada-{W}atanabe-{E}ngelbert theorem for general stochastic
  equations and inequalities.
\newblock {\em Electron. J. Probab.}, 12:951--965, 2007.

\bibitem{Mueller91}
C.~Mueller.
\newblock On the support of solutions to the heat equation with noise.
\newblock {\em Stochastics Stochastics Rep.}, 37(4):225--245, 1991.

\bibitem{MMQ2010}
C.~Mueller, L.~Mytnik, and J.~Quastel.
\newblock Effect of noise on front propagation in reaction-diffusion equations
  of kpp type.
\newblock {\em Inventiones mathematicae}, 184(2):405--453, 2011.

\bibitem{MP09}
L.~Mytnik and E.A. Perkins.
\newblock Pathwise uniquenesss for stochastic partial differential equations
  with {H}\"{o}lder coefficients.
\newblock {\em Probab. Theory Related Fields}, 149:1--96, 2011.

\bibitem{MPS06}
L.~Mytnik, E.A. Perkins, and A.~Sturm.
\newblock On pathwise uniqueness for stochastic heat equations with
  non-{L}ipschitz coefficients.
\newblock {\em Ann. Probab.}, 34(5):1910--1959, 2006.

\bibitem{Nakao1983}
S.~Nakao.
\newblock On pathwise uniqueness and comparison of solutions of one-dimensional
  stochastic differential equations.
\newblock {\em Osaka J. Math.}, 20(1):197--204, 1983.

\bibitem{NualartPardoux1994}
D.~Nualart and E.~Pardoux.
\newblock Markov field properties of solutions of white noise driven
  quasi-linear parabolic {PDE}s.
\newblock {\em Stochastics Stochastics Rep.}, 48(1-2):17--44, 1994.

\bibitem{Perk2002}
E.~Perkins.
\newblock Dawson-{W}atanabe superprocesses and measure-valued diffusions.
\newblock In {\em Lectures on probability theory and statistics
  ({S}aint-{F}lour, 1999)}, volume 1781 of {\em Lecture Notes in Math.}, pages
  125--324. Springer, Berlin, 2002.

\bibitem{Rippl-Sturm2013}
T.~Rippl and A.~Sturm.
\newblock New results on pathwise uniqueness for the heat equation with colored
  noise.
\newblock {\em Electron. J. Probab.}, 18:no. 77, 1--46, 2013.

\bibitem{rosen}
J.~Rosen.
\newblock Joint continuity of the intersection local times of {M}arkov
  processes.
\newblock {\em Ann. Probab.}, 15(2):659--675, 1987.

\bibitem{Shi94}
T.~Shiga.
\newblock Two contrasting properties of solutions for one-dimensional
  stochastic partial differential equations.
\newblock {\em Canad. J. Math.}, 46(2):415--437, 1994.

\bibitem{Walsh}
J.~B. Walsh.
\newblock An introduction to stochastic partial differential equations.
\newblock In {\em \'{E}cole d'\'et\'e de probabilit\'es de {S}aint-{F}lour,
  {XIV}---1984}, volume 1180 of {\em Lecture Notes in Math.}, pages 265--439.
  Springer, Berlin, 1986.

\bibitem{YW71}
S.~Watanabe and T.~Yamada.
\newblock On the uniqueness of solutions of stochastic differential equations.
  {II}.
\newblock {\em J. Math. Kyoto Univ.}, 11:553--563, 1971.

\end{thebibliography}

\end{document}